\documentclass[12pt, reqno]{amsart}
\usepackage{amssymb, hyperref}
\usepackage{amsthm}
\usepackage{graphicx}
\usepackage[all,2cell]{xy}
\usepackage{verbatim}
\usepackage{tikz}
\usepackage{tikz-cd}
\usepackage{hyperref}
\usetikzlibrary{matrix,arrows,decorations.pathmorphing}
\usepackage{mathrsfs}
\numberwithin{equation}{section}
\newtheorem*{theorem*}{Theorem}
\newtheorem*{corollary*}{\bf Corollary}
\newtheorem*{remark*}{\bf Remark}
\usepackage{lipsum}
\usetikzlibrary{matrix,arrows,decorations.pathmorphing}
\usepackage{lineno}
\newtheorem{theorem}{Theorem}[section]
\newtheorem{corollary}[theorem]{Corollary}
\newtheorem{definition}[theorem]{Definition}
\newtheorem{example}[theorem]{Example}
\newtheorem{lemma}[theorem]{Lemma}

\newtheorem{proposition}[theorem]{Proposition}
\newtheorem{remark}[theorem]{Remark}
\usepackage{csquotes}
\usepackage{epigraph}

\title[ On Fano and weak Fano of  BSDH varieties]
{On Fano and weak Fano\\  Bott-Samelson-Demazure-Hansen varieties}
 \author{B. Narasimha Chary}
\address{%
B. Narasimha Chary\\
Institut Fourier, UMR 5582 du CNRS\\
Universit{\'e} de Grenoble Alpes\\
 CS 40700, 38058\\
Grenoble cedex 09, France.\\
Email: narasimha-chary.bonala@univ-grenoble-aples.fr
}


\begin{document}
\maketitle
\begin{abstract} 
Let $G$ be a simple algebraic group over the field of complex numbers.
Fix a maximal torus $T$ and a Borel subgroup $B$ of $G$ containing $T$.
Let $w$ be an element of the Weyl group $W$ of $G$, and let $Z(\tilde w)$ be the Bott-Samelson-Demazure-Hansen (BSDH) variety 
corresponding to a reduced expression $\tilde w$ of $w$ with respect to 
the data $(G, B, T)$.

In this article we give complete characterization of the expressions $\tilde w$ such that the
corresponding 
 BSDH variety $Z(\tilde w)$ is Fano or weak Fano. 
As a consequence we prove vanishing theorems of 
the cohomology of tangent bundle of certain BSDH varieties 
and hence we get some local rigidity results.
\end{abstract}
\let\thefootnote\relax\footnotetext{The author is supported by AGIR Pole MSTIC project run by the University of Grenoble Alpes, France.}

  

\section{Introduction}\label{intro}

Let $G$ be a simple algebraic group over  the field $\mathbb{C}$ of 
complex numbers. Let $T$ be a maximal torus in $G$, and $B$ be a Borel subgroup of $G$ containing $T$.
Let $N_G(T)$ be the normalizer of $T$ in $G$ and let $W:=N_G(T)/T$ be the Weyl group of $G$.  
 For $w\in W$, let $X(w)$ be the corresponding Schubert variety in the flag variety $G/B$.   
Let $\tilde w=s_{\beta_1}\cdots s_{\beta_r}$ be a reduced expression of $w$,
where $\beta_k$ for $1\leq k\leq r$  are simple roots.
Let $Z(\tilde w)$ be the BSDH variety corresponding to the expression $\tilde w$, which is a natural desingularization of the Schubert variety $X(\tilde w)$.
In \cite{anderson2014effective} and \cite{Charytoric}, the authors described the divisors $D$ in $Z(\tilde w)$ such that the pair $(Z(\tilde w), D)$ is {\it log Fano} and in \cite{anderson2014schubert} the case of Schubert varieties were considered.
A smooth projective variety $X$ is called {\it \bf Fano} (respectively,  {\it \bf weak Fano}) if its 
 anti-canonical divisor $- K_X$ is 
 ample (respectively,  nef and big).
In \cite{Charytoric},
we obtained some expressions $\tilde w$ for which the BSDH variety $Z(\tilde w)$ is 
Fano or weak Fano by using the degeneration of a BSDH variety as a Bott tower which is a toric variety.
 
 In this paper, we characterize all the expressions $\tilde w$ of $w$ such that 
the BSDH variety $Z(\tilde w)$ is Fano or weak Fano. In \cite{Chary1}, we proved that when $G$ is simply laced, the higher cohomology groups of the tangent bundle of BSDH varieties 
  vanish.
 In \cite{Chary11} and \cite{Charytoric}, we obtained 
  some vanishing results of the cohomology of tangent bundle of BSDH varieties for non simply laced groups and we observed that these vanishing results depend on the given reduced expression $\tilde w$.
The case of non-reduced expressions were considered in \cite{Charynonreduced}.
  Here we prove some vanishing results for certain BSDH varieties.

  To describe the results we need some notations. Let $\tilde w=s_{\beta_1} \cdots s_{\beta_i}\cdots s_{\beta_j} \cdots s_{\beta_r}$ be a reduced expression of $w$
and let $\beta_{ij}:=\langle \beta_j, \check \beta_i \rangle$ for $j>i$, where $\check \beta_i$ is the co-root of $\beta_i$ (see Section \ref{preleminaries} for more details).
For $1\leq i \leq r$, define  
 $$\eta^+_i:=\{r\geq j> i: \beta_{ij}>0\} ~~\mbox{ and}~~ \eta^-_i:=\{r\geq j> i: \beta_{ij}<0\}.$$
  Let $s(i)$ be the least integer $j>i$ such that $\beta_{ij}>0$. 
  Define $$\eta_{i, s(i)}^-:=\{s(i)>j>i: \beta_{ij}<0\}.$$
 We write the integers $\beta_{ij}$ for $j>i$ in the matrix form:
 $$
  \begin{bmatrix}
        0 & \beta_{12} & \beta_{13} & \dots & \beta_{1r}\\
        0 & 0 &\beta_{23} & \dots & \beta_{2r}\\
        0 & 0 & 0 &\dots & \beta_{3r}\\
        \vdots & \vdots & &\ddots & \vdots \\
        0 & \dots & \dots & & 0
       \end{bmatrix}_{r\times r}
   $$
Then the cardinality $|\eta_i^+|$ (respectively,  $|\eta_i^-|$) of $\eta_i^+$ (respectively,  of $\eta_i^-$) gives the number of positive (respectively,  negative) entries in $i^{th}$
row of the matrix. If there exists a positive entry in $i^{th}$ row, 
then $|\eta_{i, s(i)}^-|$ gives the number of negative entries in $i^{th}$ row between $(i+1)^{th}$ and $(s(i)-1)^{th}$ column. 
Note that in the matrix, since $\beta_k$'s are simple roots, the positive entries are $2$ and the negative entries are $-1, -2$ or $-3$. 

Let $1\leq i \leq r$. If $|\eta_{i}^-|=1$ (respectively,  $|\eta_{i}^-|=2$ ),
 then set $\eta^-_i=\{l\}$ (respectively,   $\eta^-_i=\{l_1, l_2\}$).
 If $|\eta_{i, s(i)}^-|=1$ (respectively,  $2$), then 
  $\eta_{i, s(i)}^-:=\{m \}$ (respectively,  $\eta_{i, s(i)}^-:=\{m_1, m_2 \}$). 

  \noindent 
 \begin{itemize}
   \item $N^{I}_i$  is the condition that 
 
 Case 1: $|\eta^+_i|=0$, 
 $|\eta^-_i|\leq 1$, and if $|\eta^-_i|=1$ 
 then $\beta_{il}=-1$; or
 
  Case 2: $|\eta_i^+|=1$, $|\eta^-_{i, s(i)}|=1$ and $\beta_{im}=-1$.
  \item $N^{II}_i$  is the condition that 
  
  Case 1: Assume that $|\eta^+_i|=0$. 
  Then $|\eta_i^-|\leq 2$, and 
  if $|\eta^-_i|=1 (\mbox{respectively, }~ 2)$ then  $\beta_{il}=-1$ or $-2$ 
  (respectively,  $\beta_{il_1}=-1=\beta_{il_2}$).
  
  Case 2 : If $|\eta_i^+|\geq   1$, then 
  $|\eta_{i, s(i)}^-|\leq 2$, and if $|\eta_{i, s(i)}^-|=1$(respectively,  $2$), 
    then $\beta_{im}=-1 ~\mbox{or} -2$ (respectively,  $\beta_{im_1}=-1=\beta_{im_2}$). 
  \end{itemize}

  \begin{example} Consider the following matrices:
  
  (1) $ \begin{bmatrix}
        0 & -1& 0 & 0 & 0\\
        0 & 0 & 2 & -1 & 2\\
        0 & 0 & 0 & 0& 0\\
        0 & 0 & 0 & 0& -1\\
        0 & 0 & 0 & 0& 0
              \end{bmatrix},
 $  (2) $ \begin{bmatrix}
        0 & -1& 0 & -1 & 0\\
        0 & 0 & 0 & -2 &0\\
        0 & 0 & 0 & 0& -1\\
        0 & 0 & 0 & 0& -2\\
        0 & 0 & 0 & 0& 0
              \end{bmatrix},
 $  (3) $\begin{bmatrix}
        0 & 2& -1 & 2 & -1 & 2\\
        0 & 0 & 2 & -1 & 2 & 0\\
        0 & 0 & 0 & 0& -2 & 0 \\
        0 & 0 & 0 & 0& -1& -1\\
        0 & 0 & 0 & 0& 0 & -1\\
        0 & 0 & 0 & 0& 0 & 0
              \end{bmatrix}  
 $, \\ (4)
 $\begin{bmatrix}
        0 & 2& -1 & -1 & 2\\
        0 & 0 & 2 & -2 &2\\
        0 & 0 & 0 & -1& 0\\
        0 & 0 & 0 & 0& -1\\
        0 & 0 & 0 & 0& 0
              \end{bmatrix}$
           and (5)
 $\begin{bmatrix}
        0 & -1 & -1 & -1 &  0\\
        0 & 0 & -1 & 0 &-1\\
        0 & 0 & 0 & 0& -3\\
        0 & 0 & 0 & -2& 0\\
        0 & 0 & 0 & 0& 0
              \end{bmatrix}  
 $. 
 
 Then,  in the matrix (1), the condition $N_i^I$ holds for all $1\leq i \leq 5$. For the matrices (2), (3) and (4), the condition $N_i^{II}$ holds for all $1\leq i\leq 5$.
For the matrix (5),  $N_2^{II}$ and $N_4^{II}$ holds, but $N_1^{II}$, $N_2^{I}$, $N_3^{II}$ and  $N_4^{I}$ does not hold. 

\end{example}

   \begin{definition}   
 We say that $Z(\tilde w)$ satisfies 
 \underline{condition $I$} (respectively, \underline{condition $II$}) if
 $N^{I}_i$ (respectively, $N_i^{II}$) holds for all $1\leq i\leq r$.
   \end{definition}
 Note that for all $1\leq i\leq r$, 
 $N_i^I\Longrightarrow N_i^{II} $. Hence if $Z(\tilde w)$ satisfies condition $I$, then it also satisfies the condition $II$.
 Also note that in the condition $N_i^I$,  if $|\eta^+_i|=0$ and 
 $|\eta^-_i|=0$ for all $1\leq i \leq r$, then the corresponding matrix is zero matrix and $Z(\tilde w)\simeq \mathbb P^1\times\cdots \times  \mathbb P^1$ ($r$-times). 
    
    We prove the following:
 \begin{theorem*}[see Theorem \ref{bsdhfano}]\
 
 \begin{enumerate}
   \item 
   $Z(\tilde w)$ is Fano if and only if it satisfies  $I$.
\item   $Z(\tilde w)$ is weak  Fano if and only if it satisfies  $II$.
  \end{enumerate}
  \end{theorem*}
 
   Recall that a {\it \bf Coxeter type element} is an element of the Weyl group 
having a reduced expression of the form $s_{\beta_1}s_{\beta_2} \cdots s_{\beta_r}$, $r\leq rank(G)$
such that $\beta_{j}\neq \beta_{k}$ whenever $j\neq k$ (see \cite[Section 4.4, page 56]{Humconjugacy}).
Note that for a Coxeter type element $w$, it is known that the Schubert variety $X(w)$ is a toric variety and hence $Z(\tilde w)$ is also a toric variety (see \cite{karuppuchamy2013schubert}).
By using a vanishing result for the cohomology of tangent bundle on Fano toric varieties and using the above theorem we get the following result.   Let $T_{Z(\tilde w)}$ be the tangent bundle of $Z(\tilde w)$. Then, 
 \begin{corollary*}[see Corollary \ref{vanishing} and Corollary \ref{rigid}]
   If $\tilde w$ is Coxeter type element and $Z(\tilde w)$ satisfies $I$, then $H^j(Z(\tilde w), T_{Z(\tilde w)})=0$ for all $j>0$. In particular, $Z(\tilde w)$ is locally rigid.
 \end{corollary*}

 The organization of the article: In Section 2, we set up the notations and recall the construction of BSDH varieties.
 Section 3, contains the results on Mori cone of the BSDH variety. Section 4, contains some results on canonical line bundle of BSDH variety.
In Section 5 and 6, we prove the main theorem and some applications. 
  
 \section{Preliminaries}\label{preleminaries}
 Let $G$ be a simple algebraic group over the field of complex numbers.
 Fix a maximal torus $T$ and $B$ a Borel subgroup containing $T$.
 We denote by $R$ the set of all roots, by $R^+$ (respectively,  $R^-$) the set of positive (respectively,  negative) roots, by $S$ the set of simple roots
 corresponding to the data $(G, B, T)$ and by $W$ the associated Weyl group.
Let $S=\{\alpha_1, \ldots, \alpha_n\}$, where $n$ is the rank of $G$.
The simple reflection corresponding to the simple root $\alpha$ is denoted by $s_{\alpha}$.
For $w\in W$, let $\tilde w=s_{\beta_1}\cdots s_{\beta_r}$ be a reduced expression of $w$ in simple reflections.
For $\alpha\in S$, we denote by $P_{\alpha}$ the minimal parabolic subgroup of $G$ 
generated by $B$ and a representative of $s_{\alpha}$.

We recall that the Bott-Samelson-Demazure-Hansen (BSDH) variety corresponding to a reduced expression $\tilde w=s_{\beta_1}\cdots s_{\beta_r}$ is defined by 
$$Z(\tilde w)=P_{\beta_1}\times \cdots \times P_{\beta_r}/B^r,$$
where the action of $B^r$ on $P_{\beta_1}\times \cdots \times P_{\beta_r}$ is given by 
$$(p_1, \ldots, p_r)\cdot (b_1, \ldots, b_r)=(p_1\cdot b_1, b_1^{-1}\cdot p_2\cdot b_2, \ldots, b_{r-1}^{-1}\cdot p_r \cdot b_r),$$ 
$p_j\in P_{\beta_j}, b_j \in B ~\mbox{ for all}~1\leq j \leq r$ (see \cite[p.73, Definition 1]{brion} and \cite{demazure1974}).

The variety $Z(\tilde w)$ is smooth projective and it is a desingularization of the Schubert variety $X(w)(:=\overline{BwB/B} \subset G/B)$ (see \cite{demazure1974}).
We denote the natural birational surjective morphism from 
$Z(\tilde w)$ to $X(w)$ and the composition map $Z(\tilde w)\to X(w)\hookrightarrow G/B$ by $\phi_w$. 
Let $f_r:Z(\tilde w)\longrightarrow Z(\tilde w')$ denote the map induced by the projection
$$P_{\beta_1}\times \cdots \times P_{\beta_r} \longrightarrow P_{\beta_1}\times \cdots \times P_{\beta_{r-1}}, $$
where $\tilde w'=s_{\beta_1}\cdots s_{\beta_{r-1}}$. Note that $f_{r}$ is a $\mathbb P^1$-fibration and we have the following cartesian diagram (see \cite[Page 66]{brion}):
\begin{equation}\label{cartesian}
 \xymatrix{
Z(\tilde w)=Z(\tilde w')\times _{G/P_{\beta_r}}G/B \ar[d]^{f_r} \ar[rrr]^{\phi_w} &&& G/B\ar[d]^{f}\\
Z(\tilde w') \ar[rr]^{\phi_{w'}} &&& G/P_{\beta_r}
}
\end{equation}
where $f:G/B \longrightarrow G/P_{\beta_r}$ is given by $gB\mapsto gP_{\beta_r}$.
Let $\sigma_r: Z(\tilde w')\longrightarrow Z(\tilde w)$ denote the map induced by the inclusion
$P_{\beta_1}\times \cdots \times P_{\beta_{r-1}} \longrightarrow P_{\beta_1}\times \cdots \times P_{\beta_{r}}$ which takes $(p_1, \ldots ,p_{r-1})$ into $(p_1,\ldots, p_{r-1}, 1)$.
Then $\sigma_r$ is a closed immersion.
Let $1\leq i \leq r$. We denote by $Z_i$, the divisor in $Z(\tilde w)$ given by 
$$Z_i:=\{[(p_1, \ldots, p_r)]\in Z(\tilde w): p_i\in B\}.$$

Let $\mathfrak{g}$ (respectively,  $\mathfrak{h}$) denote the Lie algebra of $G$ (respectively,  $T$). Let $X(T)$ denote the group of all characters of $T$.
We have $X(T)\otimes \mathbb R=Hom(\mathfrak{h}_{\mathbb R}, \mathbb R)$, the dual of the real form $\mathfrak{h}_{\mathbb R}$ of $\mathfrak{h}$.
We denote $\langle ~, ~\rangle$ the positive definite $W$-invariant form on $Hom(\mathfrak{h}_{\mathbb R}, \mathbb R)$ induced by the Killing form of $\mathfrak{g}$.
For more details we refer to \cite{Hum1} and also see \cite{springer2010linear}. We write $\beta_{ij}=\langle \beta_j , \check \beta_i \rangle=\langle \beta_j    , 2\beta_i\rangle/\langle \beta_i, \beta_i \rangle $ for simple roots $\beta_i$ and $\beta_j$.

\section{Mori cone of a BSDH variety}
First we recall Mori cone of a smooth projective variety.
Let $X$ be a smooth projective variety, let $Z_1(X):=\{1\mbox{-cycles of}~ X\}$, where $1$-cycle of $X$ is a formal sum $\sum_{i}a_iC_i$
with irreducible curves $C_i$ and $a_i\in \mathbb Z$. 
Consider the $\mathbb R$-vector spaces 
$$N_1(X):=(Z_1(X)/\equiv )\otimes \mathbb R ~\mbox{and}~ N^1(X):=(CDiv(X)/\equiv)\otimes \mathbb R,$$
where $\equiv$ denotes numerical equivalence and $CDiv(X)$ is the group of (Cartier) divisors in $X$. 
It is known that $N_1(X)$ and $N^1(X)$ are dual vector spaces via natural paring and are finite dimensional (see \cite[Chapter IV]{kleiman1966toward}).
Define $NE(X)$ in $N_1(X)$ by 
$$NE(X):=\{\sum_{finite}a_iC_i: a_i\in \mathbb R_{\geq 0} ~\mbox{and}~ C_i ~\mbox{is an irreducible curve in $X$}~\},$$
 the real cone in $N_1(X)$ generated by classes of irreducible curves.
The Mori cone $\overline{NE}(X)$ of $X$ is the closure of $NE(X)$ in $N_1(X)$.
Let $Nef(X)$ (respectively,  $Amp(X)$) be the cone generated by classes of numerically effective (nef) (respectively,  ample) divisors in $N^1(X)$.
Then $Nef(X)$ (respectively,  $Amp(X)$) is called {\it nef cone} (respectively,  {\it ample cone}) of $X$. 
By Kleiman's criterion, the closure of $Amp(X)$ in $N^1(X)$ is $Nef(X)$ (see \cite[Section 1, Chapter IV]{kleiman1966toward} and also \cite[Theorem 1.4.23]{lazarsfeld2004positivity}).
Also the nef cone $Nef(X)$ is dual to the Mori cone $\overline{NE}(X)$
(see \cite[Section 2, Chapter IV]{kleiman1966toward} and \cite[Page 61]{lazarsfeld2004positivity}). 

We recall some results on Mori cone of BSDH varieties from 
 \cite{perrin2005rational} (see also \cite{perrin2007}).
  Denote $Z_i=Z(s_{\beta_1}\cdots \widehat{s_{\beta_i}}\cdots s_{\beta_r})$, where $\widehat{a}$ means $a$ is removed (see Section 2.1).
 Note that 
  $$Z_{i}=f_r^{-1}\cdots f_{i+1}^{-1}(Im(\sigma_i))$$ and it is a divisor in $Z(\tilde w)$. For any $K\subset [1, r]:=\{1,\ldots, r\}$,
 denote  
$$Z_{K}=\bigcap_{i\in K}Z_{i}$$
 Note that the codimension of $Z_{K}$ is $|K|.$ The classes of $Z_{K}$ form a basis
of $N_*(Z(\tilde w))$ and M. Demazure completely described the
 structure of this group (see \cite[\S 4, Proposition 1]{demazure1974}).
 In \cite{lauritzen2002line} and in \cite{anderson2014effective}, the authors studied effective divisors in $Z(\tilde w)$.

 For $1\leq i \leq r$, define the curve $C_i=Z_{[1, r]\setminus  i}$.
 These curves $C_i$ for $1\leq i \leq r$ form a basis of $N_{1}(Z(\tilde w))$ and also we have
   $$[C_i]=\prod_{j\neq i}[Z_{j}].$$ 
 Define 
$$[\tilde{C_i}]:=[C_i]-[C_{s(i)}],$$ 
where $s(i)$ is the least integer $j>i$ 
such that $\beta_{ij}>0$. Set $[C_{s(i)}]:=0$, if there is no such $s(i)$.
Then we have the following result from \cite[Proposition 4.2]{perrin2005rational} (see also \cite[Corollary 2.15]{perrin2007}):

\begin{proposition}\label{perrin}\

 \begin{enumerate}
  \item The classes $[\tilde{C_i}]$ generate $N_{1}(Z(\tilde w))$.
  \item $\overline{NE}(Z(\tilde w))=\sum_{i=1}^{r}\mathbb R_{\geq 0}[\tilde{C_i}]$ and $[\tilde{C_i}]'s$ are extremal rays.
 \end{enumerate}

\end{proposition}

\section{The canonical line bundle of $Z(\tilde w)$}

It is well known that the canonical line bundle $\mathcal O_{Z(\tilde w)}(K_{Z(\tilde w)})$ 
is given by 
$$\mathcal O_{Z(\tilde w)}(K_{Z(\tilde w)})=\mathcal O_{Z(\tilde w)}(-\partial Z(\tilde w))\otimes \mathcal L(-\delta),$$
where $\partial Z(\tilde w)$ is the boundary divisor and  
$\delta$ is half sum of the positive roots (see \cite[Proposition 2]{mehta1985frobenius}).  
By Cartesian diagram \ref{cartesian} (page 4), we can see that the relative tangent bundle of $f_r:Z(\tilde w)\to Z(\tilde w')$ is the homogeneous line bundle 
 $\mathcal L(\beta_r)$ corresponding to the root $\beta_r$. We denote $D_r$, the corresponding divisor in $Z(\tilde w)$.
 By abuse of notation $D_j$ also denote the pullback to $Z(\tilde w)$  for all $1\leq j \leq r-1$. 
Now we have,
\begin{lemma}\label{canonical}
 The canonical line bundle $\mathcal O_{Z(\tilde w)}(K_{Z(\tilde w)})$ of $Z(\tilde w)$ is given by 
  $$\mathcal O_{Z(\tilde w)}(K_{Z(\tilde w)})= \mathcal O_{Z(\tilde w)}(-\sum_{j=1}^{r} D_j).$$
\end{lemma}
\begin{proof}
 Proof is by induction on $r$. If $r=1$, then we are done.
 Assume that $r>1$ and the result is true for $r-1$.
 Since $f_r:Z(\tilde w)\to Z(\tilde w')$ is a $\mathbb P^1$-fibration, the relative canonical bundle $K_{f_r}$ is given by 
  $$K_{f_r}=\mathcal O_{Z(\tilde w)}(K_{Z(\tilde w)})\otimes f_{r}^*( 
 \mathcal O_{Z(\tilde w')}(\check K_{Z(\tilde w')}))$$
 (see \cite[Corollary 24, page 56]{kleiman1980relative}).
 Then, by induction we see that  
 $$\mathcal O_{Z(\tilde w)}(K_{Z(\tilde w)})=K_{f_r}\otimes f_r^*(\mathcal O_{Z(\tilde w')}(-\sum_{j=1}^{r-1} D_j)).$$
 Since $K_{f_r}=\mathcal O_{Z(\tilde w)}(-T_{r})$, we have  
 \[ K_{Z(\tilde w)}= \mathcal O_{Z(\tilde w)}(-\sum_{j=1}^{r} D_j).\qedhere\]
\end{proof}

We have the following.
\begin{lemma}\label{5.10} For all $1\leq i\leq r$, we have  
$$-K_{Z(\tilde w)}\cdot [C_i]=2+\sum_{j>i}\beta_{ij}.$$
 \end{lemma}
\begin{proof}
 By \cite[Proposition 3.11]{perrin2005rational}, we have the following:
 $$[C_i]\cdot [D_j]=\begin{cases}
                           0 & ~~\mbox{if}~~ i>j\\
                             \beta_{ij} & ~\mbox{if}~i\leq j
                          \end{cases}
$$
 Then by Lemma \ref{canonical}, we see that 
 $$-K_{Z(\tilde w)}\cdot [C_i]=\sum_{j\geq i}\beta_{ij}.$$
 Since $\beta_{ii}=\langle \beta_i , \check \beta_i\rangle=2$, we get \[-K_{Z(\tilde w)}\cdot [C_i]=2+\sum_{j>i}\beta_{ij}. \qedhere \]
\end{proof}

\section{Fano or weak Fano BSDH varieties}
 In this section we study the Fano and weak Fano properties for BSDH varieties. 
 We refer to \cite{lazarsfeld2004positivity} for more details on ample, big and nef divisors. 
 First we prove the following lemma
 (see \cite[Section 6]{Charymori} and also see \cite[Chapter II]{hartshorne2006ample}):
    \begin{lemma}\label{big1}
 Let $X$ be a smooth projective variety and let $U$ be an open affine subset of $X$. Let 
 $D$ be an effective divisor with support $X\setminus U$. Then $D$ is {\it big}.
   \end{lemma}
\begin{proof}
It suffices to show that there exists an effective divisor $E$ with support $X\setminus U$ such that $E$ is {\it big}.
Indeed, we then have $mD=E+F$ for some $m\geq 0$ and for some effective divisor $F$.
Then $E+F$ is big and hence so is $D$.

Assume that 
$dim(X)=n$. Then there exist algebraically independent elements $f_1,\ldots, f_n$ in $\mathcal O_X(U)$ over $\mathbb C$. 
 View $f_1,\ldots, f_n$ as rational functions on $X$, then $f_1,\ldots, f_n\in H^0(X, \mathcal O_X(E))$
for some effective divisor $E$ with support $X\setminus U$ (since $div(f_i)$ is an effective divisor with support in $X\setminus U$ for $1\leq i \leq r$).
Thus, the monomials in $f_1,\ldots, f_n$ of any degree $m$ are linearly independent elements of $H^0(X, \mathcal O_X(mE))$. 
So $dim(H^0(X, \mathcal O_X(mE)))$ grows like $m^n$ as $m\to \infty$.
Hence $E$ is {\it big} (see \cite[Corollary 2.1.38 and Lemma 2.2.3]{lazarsfeld2004positivity}) and this completes the proof.
\end{proof}

We get the following as a variant of Lemma \ref{big1}.

    \begin{lemma}\label{big}
Let $X$ be a smooth projective variety and $D$ be an effective divisor. Let $supp(D)$ denotes the support of $D$.
If $X\setminus supp(D)$ is affine, then $D$ is {\it big}. 
\end{lemma}
We use the notation and terminology as in Section \ref{intro}. 
 Then we have:
 \begin{theorem}\label{bsdhfano}\
  
  \begin{enumerate}
\item  $Z(\tilde w)$ is Fano if and only if it satisfies $I$.
  \item 
    $Z(\tilde w)$ is weak  Fano if and only if it satisfies $II$.
    \end{enumerate}

  \end{theorem}
  \begin{proof} Proof of (2): 
  Since $\tilde w$ is a reduced expression, 
  $Z(\tilde w)$ has an open $B$-orbit (which is affine), then by Lemma \ref{big}
   we conclude that 
  $-K_{Z(\tilde w)}$ is big. 
    By Proposition \ref{perrin} we have 
      $$\overline{NE}(Z(\tilde w))=\sum_{i=1}^{r}\mathbb R_{\geq 0}[\tilde{C_i}].$$
  Note that the nef cone $Nef(Z(\tilde w))$ is the dual cone of $\overline{NE}(Z(\tilde w))$ (see \cite[Proposition 1.4.28]{lazarsfeld2004positivity}).
   Hence    
   \begin{equation}\label{above}
       -K_{Z(\tilde w)} ~\mbox{is nef if and only if}~ -K_{Z(\tilde w)}\cdot [\tilde{C_{i}}]\geq 0 ~\mbox{for all}~ 1\leq i\leq r.
      \end{equation}

    We prove by using Lemma \ref{5.10} that,     
  \[-K_{Z(\tilde w)}~\mbox{is nef if and only if}~ Z(\tilde w) ~\mbox{satisfies}~ II.  
    \]
 Assume that $-K_{Z(\tilde w)}$ is nef. Then by (\ref{above}) we have 
 $$-K_{Z(\tilde w)}\cdot [\tilde{C_i}]\geq 0 ~\mbox{for all }~1\leq i \leq r.$$
 Fix $1\leq i\leq r$.
 
  \underline{Case 1}: Assume that there exists no such $s(i)$.  Then  $|\eta_i^+|=0$ and  $[\tilde{C_{i}}]=[C_i]$. Hence 
 $$-K_{Z(\tilde w)}\cdot [C_i]\geq 0.$$
 Then, by Lemma \ref{5.10} we get $$2+\sum_{j>i}\beta_{ij}\geq 0.$$
 Since there is no $s(i)$, $\beta_{ij}\leq 0$ for all $j>i.$ Hence 
 we have $$-2\leq \sum_{j>i}\beta_{ij} \leq 0. $$
 
 \underline{Subcase (i):} If $\sum_{j>i}\beta_{ij}=0$, then $\beta_{ij}=0$ for all $j>i$. Hence by definition of $\eta_{i}^-$,
we have  $|\eta_i^-|=0$. Therefore, $Z(\tilde w)$ satisfies condition $N_i^I$.

 \underline{Subcase (ii):}  If $\sum_{j>i}\beta_{ij}=-1$, then  $|\eta_i^-|=1$ and $\beta_{il}=-1$. Hence $Z(\tilde w)$ satisfies condition $N_i^I$.

 \underline{Subcase (iii):}  If $\sum_{j>i}\beta_{ij}=-2$, then $|\eta_i^-|=1$ or $2$. 
 If $|\eta_i^-|=1$, then $\beta_{il}=-2$. If  $|\eta_i^-|=2$, then $\beta_{il_1}=-1=\beta_{il_2}$. Hence
 $Z(\tilde w)$ satisfies condition $N_i^{II}$. 
 
 \underline{Case 2}: If there exists $s(i)$, then $[\tilde{C_i}]=[C_i]-[C_{s(i)}]$.
  By Lemma \ref{5.10}, we get $$-K_{Z(\tilde w)}\cdot [\tilde{C_i}]=\sum_{j>i}\beta_{ij}-\sum_{j>s(i)}\beta_{s(i)j}.$$
    Since $\beta_{ij}=\beta_{s(i)j}$ for $j>s(i)$ and $\beta_{is(i)}=2$, we see that 
   $$-K_{Z(\tilde w)}\cdot [\tilde{C_i}]=2+\sum_{s(i)> j>i}\beta_{ij}.$$
   Since $-K_{Z(\tilde w)}$ is nef, we get $$\sum_{s(i)> j>i}\beta_{ij}\geq -2.$$
By definition of $s(i)$, $\beta_{ij}\leq 0$ for $i<j<s(i)$. 
Since  $\tilde w $ is a reduced expression, we have 
 \begin{equation}\label{11}
  -2\leq \sum_{s(i)>j>i}\beta_{ij} < 0. 
 \end{equation}
 Then $$1\leq |\eta_{i, s(i)}^-|\leq 2.$$
 
\noindent \underline{Subcase (i):} If $|\eta_{i, s(i)}^-|=1$, then by (\ref{11}), $\beta_{im}=-1$ or $-2$.\\
 \underline{Subcase (ii):} If $|\eta_{i, s(i)}^-|=2$, then by (\ref{11}), $\beta_{im_1}=-1=\beta_{im_2}$.

 Hence $Z(\tilde w)$ satisfies $II$. 
 Thus, if $-K_{Z(\tilde w)}$ is nef, then $Z(\tilde w)$
 satisfies $II$.  
   Hence we conclude that if  $Z(\tilde w)$ is weak  Fano then $Z(\tilde w)$ satisfies $II$.
    Similarly,  we can prove that if $Z(\tilde w)$ satisfies $II$,
 then $Z(\tilde w)$ is weak Fano.
 
 Proof of (1): 
  Since the ample cone $Amp(Z(\tilde w))$ is interior of the nef cone $Nef(Z(\tilde w))$, (1) follows from  
  Nakai-Kleiman criterion for ampleness (see \cite[Theorem 1.2.23]{lazarsfeld2004positivity}) and by using the similar arguments as in the proof of $(2)$.   
This completes the proof of the theorem.   
   \end{proof}

   \begin{example} Let $G=SL(5, \mathbb C)$. We use the notation from \cite{Hum1} and \cite{Hum2}.
      \begin{enumerate}
 \item Let $\tilde w=s_{1}s_3$. 
 Then  $|\eta^+_{i}|=|\eta^-_{i}|=0$ for $1\leq i \leq 2$ and $Z(\tilde w)\simeq \mathbb P^1\times \mathbb P^1$. 
 Hence $Z(\tilde w)$ satisfies $I$. 
 \item Let $\tilde w=s_1s_{2}s_{3}$. Then $|\eta^+_{i}|=0$ and $|\eta^-_{i}|=1$ for $i=1, 2$, also we have 
 $\beta_{12}=\beta_{23}=-1$. So, $Z(\tilde w)$ satisfies $I$ and by above theorem, $Z(\tilde w)$ is Fano.
 
 \item Let $\tilde w=s_{1}s_{2}s_{1}$. Then $|\eta^+_{1}|=1, |\eta^+_{2}|=0$, $|\eta^-_{1}|=|\eta^-_{2}|=1$, and $\beta_{12}=\beta_{23}=-1$, and so $Z(\tilde w)$ satisfies $I$. Hence $Z(\tilde w)$ is Fano.

  \item Let $\tilde w=s_2s_3s_1s_2$. Then $|\eta_{1}^+|=1$ , $|\eta_{i}^+|=0$ for $i=2,3$, $|\eta_{i}^-|=1$ for $i=2, 3$, and $s(1)=4$, $|\eta_{1, 4}^-|=2$; $\beta_{12}=\beta_{24}=\beta_{34}=-1$.
So, $Z(\tilde w)$ satisfies $II$ but not $I$. Hence $Z(\tilde w)$ is weak Fano but not Fano. 

 \end{enumerate}
   \end{example}
\begin{example}
 Let $G=SO(7, \mathbb C)$ (i.e. $G$ is of type $B_3$).
 \begin{enumerate}
  \item Let $\tilde w=s_2s_3$. Then $|\eta_1^+|=0$, $|\eta_1^-|=1$ and $\beta_{12}=-1$, so $Z(\tilde w)$ satisfies $I$. Hence $Z(\tilde w)$ is Fano.
  \item Let $\tilde w=s_3s_2$. Then $|\eta_1^+|=0$, $|\eta_1^-|=1$ and $\beta_{12}=-2$, so $Z(\tilde w)$ satisfies $II$ but not $I$. Hence $Z(\tilde w)$ is weak Fano but not Fano.
  \item Let $\tilde w=s_2s_3s_1s_2$. Then $|\eta_{1}^+|=1$, $|\eta_{i}^+|=0$ for $i=2,3$, $|\eta_{i}^-|=1$ for $i=2, 3$, and $s(1)=4$, $|\eta_{1, 4}^-|=2$; $\beta_{12}=\beta_{34}=-1$ and $\beta_{24}=-2$.
So $Z(\tilde w)$ does not satisfy $I$, but it satisfies $II$, by above theorem, it is weak Fano but not Fano.

 \end{enumerate}
\end{example}

\begin{example}
 Let $G$ be the group of type $G_2$ and  let $\tilde w=s_{2}s_1$. 
 Then $|\eta_1^+|=0$, $|\eta_1^-|=1$ and $\beta_{12}=-3$.
 Hence $Z(\tilde w)$ does not satisfy $II$ and by above theorem, it is not weak Fano.
\end{example}

   \section{Local rigidity of BSDH varieties}
   In \cite{Chary1}, \cite{Chary11}, we obtained some vanishing theorems on cohomology of tangent bundle of BSDH variety
(see \cite[Proposition 3.1]{Chary1} and \cite[Theorem 8.1]{Chary11}) and in \cite{Charynonreduced} the case when the expression
$\tilde w$ is non-reduced is considered.
In this section we prove some vanishing results on the
cohomology of tangent bundle of certain BSDH varieties.
 Let $T_{Z(\tilde w)}$ denote the tangent bundle of $Z(\tilde w)$.
   We have,
   \begin{corollary}\label{vanishing}
    If $\tilde w$ is a Coxeter type element and $Z(\tilde w)$ satisfies $I$, then $H^j(Z(\tilde w), T_{Z(\tilde w)})=0$ for all $j>0$.
   \end{corollary}
\begin{proof}
 First note that if $\tilde w$ is a Coxeter type element then the Schubert variety $X(w)$ is toric (see \cite{karuppuchamy2013schubert}).
Since $\tilde w$ is a reduced expression, $\phi_{w}: Z(\tilde w)\to X(w)$ is a $B$-equivariant birational morphism, it follows that $Z(\tilde w)$ is also a toric variety. 
 By Theorem \ref{bsdhfano}, if $Z(\tilde w)$ satisfies $I$, then $Z(\tilde w)$ is Fano.
 By \cite[Proposition 4.2]{Bien1996}, we see that
 $H^i(Z(\tilde w), T_{Z(\tilde w)})=0$ for all $j>0$.
\end{proof}
 
 Then we have, 
 
 \begin{corollary}\label{rigid}
  If $\tilde w$ is Coxeter type element and $Z(\tilde w)$ satisfies $I$, then $Z(\tilde w)$ is locally rigid (i.e, it does not have local deformations).
 \end{corollary}
\begin{proof}
 Since the elements of $H^1(Z(\tilde w), T_{Z(\tilde w)})$ bijectively correspond to the local deformations (see \cite[p.272, Proposition 6.2.10]{huybrechts}),
by Corollary \ref{vanishing}, we conclude that $Z(\tilde w)$ has no local deformations.
 \end{proof}

   \begin{remark}
    Note that the proof of \cite[Proposition 4.2]{perrin2005rational} works for the Kac-Moody setting (see \cite{perrin2005rational}), and hence we expect our results can be viewed in that setting also. 
       \end{remark}

{\bf Acknowledgements:}
I would like to thank Michel Brion  for valuable discussions 
and critical comments. I thank S.K. Pattanayak for careful reading of this article and the referees for their valuable suggestions.


\begin{thebibliography}{94}

\bibitem[And14]{anderson2014effective}
D.~Anderson, \emph{Effective divisors on {B}ott-{S}amelson varieties}, preprint, arXiv:1501.00034 (2014).

\bibitem[AS14]{anderson2014schubert}
D.~Anderson and A.~Stapledon, \emph{Schubert varieties are log {F}ano over the
  integers}, Proceedings of the American Mathematical Society \textbf{142}
  (2014), no.~2, 409--411.

\bibitem[BB96]{Bien1996}
F.~Bien and M.~Brion, \emph{Automorphisms and local rigidity of regular
  varieties}, Compositio Mathematica \textbf{104} (1996), no.~1, 1--26.

\bibitem[BK07]{brion}
M.~Brion and S.~Kumar, \emph{Frobenius splitting methods in geometry and
  representation theory}, vol. \textbf{231}, Springer Science \& Business Media, 2007.

\bibitem[Ch1]{Charymori}
B.N. Chary, \emph{On {M}ori cone of {B}ott towers}, preprint, arxiv.org/abs/1706.02139.

\bibitem[Ch2]{Charytoric}
\bysame, \emph{A note on toric degeneration of a
  {B}ott-{S}amelson-{D}emazure-{H}ansen variety}, preprint, arxiv.org/abs/1710.06300.


\bibitem[CK17]{Chary11}
B.N. Chary and S.S. Kannan, \emph{Rigidity of a
  {B}ott-{S}amelson-{D}emazure-{H}ansen variety for {$PSp(2n, \mathbb C)$}},
  Journal of {L}ie {T}heory \textbf{27} (2017), 435--468.

\bibitem[CKP]{Charynonreduced}
B.N. Chary, S.S. Kannan and A.J. Parameswaran, \emph{Automorphism group of a
  {B}ott-{S}amelson-{D}emazure-{H}ansen variety for non reduced case}, in
  preparation.

\bibitem[CKP15]{Chary1}
\bysame, \emph{Automorphism group of a {B}ott-{S}amelson-{D}emazure-{H}ansen
  variety}, Transformation Groups \textbf{20} (2015), no.~3, 665--698.

\bibitem[Dem74]{demazure1974}
M.~Demazure, \emph{D{\'e}singularisation des vari{\'e}t{\'e}s de {Schubert}
  g{\'e}n{\'e}ralis{\'e}es}, Annales scientifiques de l'{\'E}cole Normale
  Sup{\'e}rieure, \textbf{7} (1974), 53--88.


\bibitem[Har70]{hartshorne2006ample}
R.~Hartshorne, \emph{Ample subvarieties of algebraic varieties, {L}ecture
  {N}otes in {M}athematics}, vol. \textbf{156}, Springer, 1970.

\bibitem[Hum72]{Hum1}
J.E. Humphreys, \emph{Introduction to {Lie} algebras and representation
  theory}, vol. \textbf{9}, Springer Science \& Business Media, 1972.

\bibitem[Hum75]{Hum2}
\bysame, \emph{Linear algebraic groups}, 1975.

\bibitem[Hum11]{Humconjugacy}
\bysame, \emph{Conjugacy classes in semisimple algebraic groups}, vol.~\textbf{43},
  American Mathematical Society, 2011.

\bibitem[Huy06]{huybrechts}
D.~Huybrechts, \emph{Complex geometry: an introduction}, Springer Science \&
  Business Media, 2006.

  \bibitem[Kar13]{karuppuchamy2013schubert}
P.~Karuppuchamy, \emph{On {S}chubert varieties}, Communications in {A}lgebra
  \textbf{41} (2013), no.~4, 1365--1368.

  
\bibitem[Kle66]{kleiman1966toward}
S.L. Kleiman, \emph{Toward a numerical theory of ampleness}, Annals of
  {M}athematics (1966), 293--344.

\bibitem[Kle80]{kleiman1980relative}
\bysame, \emph{Relative duality for quasi-coherent sheaves}, Compositio
  {M}athematica \textbf{41} (1980), no.~1, 39--60.

\bibitem[Laz04]{lazarsfeld2004positivity}
R.K. Lazarsfeld, \emph{Positivity in algebraic geometry {I}: Classical setting:
  line bundles and linear series}, vol. \textbf{48}, Springer Science \& Business Media,
  2004.

\bibitem[LT04]{lauritzen2002line}
N.~Lauritzen and J.~F. Thomsen, \emph{Line bundles on {B}ott-{S}amelson
  varieties}, Journal of Algebraic Geometry \textbf{13} (2004), 461--473.

\bibitem[MR85]{mehta1985frobenius}
V.B. Mehta and A.~Ramanathan, \emph{Frobenius splitting and cohomology
  vanishing for {Schubert} varieties}, Annals of Mathematics, \textbf{122}
  (1985), no.~1, 27--40.

\bibitem[Per05]{perrin2005rational}
N.~Perrin, \emph{Rational curves on minuscule {S}chubert varieties}, Journal of
  Algebra \textbf{294} (2005), no.~2, 431--462.

\bibitem[Per07]{perrin2007}
N.~Perrin, \emph{Small resolutions of minuscule {S}chubert varieties},
  Compositio Mathematica \textbf{143} (2007), no.~5, 1255--1312.

\bibitem[Spr10]{springer2010linear}
T.~A. Springer, \emph{Linear algebraic groups}, Springer Science \& Business
  Media, 2010.

  
\end{thebibliography}
\end{document}